\documentclass[11pt]{article}
\usepackage{fullpage}
\usepackage{microtype}
\usepackage{graphicx}
\usepackage{hyperref}

\usepackage{amsmath,amsfonts,bm}

\def\1{\bm{1}}

\def\va{{\bm{a}}}
\def\vb{{\bm{b}}}

\def\ve{{\bm{e}}}

\def\vs{{\bm{s}}}

\def\vu{{\bm{u}}}
\def\vv{{\bm{v}}}
\def\vw{{\bm{w}}}
\def\vx{{\bm{x}}}
\def\vy{{\bm{y}}}

\def\mA{{\bm{A}}}
\def\mB{{\bm{B}}}
\def\mC{{\bm{C}}}

\def\mF{{\bm{F}}}
\def\mG{{\bm{G}}}
\def\mH{{\bm{H}}}
\def\mI{{\bm{I}}}
\def\mJ{{\bm{J}}}

\def\mP{{\bm{P}}}

\def\mR{{\bm{R}}}
\def\mS{{\bm{S}}}

\def\mU{{\bm{U}}}

\def\mW{{\bm{W}}}

\def\mY{{\bm{Y}}}

\DeclareMathAlphabet{\mathsfit}{\encodingdefault}{\sfdefault}{m}{sl}
\SetMathAlphabet{\mathsfit}{bold}{\encodingdefault}{\sfdefault}{bx}{n}

\def\gS{{\mathcal{S}}}

\def\sP{{\mathbb{P}}}

\def\sR{{\mathbb{R}}}

\DeclareMathOperator*{\argmax}{arg\,max}

\def\tr{\mathrm{tr}}
\usepackage[utf8]{inputenc} 
\usepackage[T1]{fontenc}    
\usepackage{url}            
\usepackage{booktabs}       
\usepackage{nicefrac}       
\usepackage{color}          
\usepackage{mathrsfs}
\usepackage{algorithm}
\usepackage{algorithmic}
\usepackage{bbm}
\usepackage{bm}
\usepackage{amsfonts, amsmath, amssymb} 
\usepackage{float}
\usepackage{diagbox}
\usepackage{subcaption}
\usepackage{breakcites}
\usepackage[numbers]{natbib}
\usepackage[dvipsnames]{xcolor}
\usepackage{grffile}
\usepackage{makecell}
\usepackage{multirow}
\usepackage{wrapfig}
\usepackage{amsthm}

 \usepackage{indentfirst} 

\newtheorem{thm}{Theorem}[section]
\newtheorem{defn}[thm]{Definition}
\newtheorem{lemma}[thm]{Lemma}
\newtheorem{assume}[thm]{Assumption}
\newtheorem{remark}[thm]{Remark}

\newtheorem{prop}[thm]{Proposition}

\def\diag{\mathrm{diag}}
\newcommand{\norm}[1]{\left\|#1\right\|}
\newcommand{\dotprod}[1]{\left\langle #1\right\rangle}
\def\tr{\mathrm{tr}}

\def\cO{\mathcal{O}}
\def\cN{\mathcal{N}}
\def\EE{\mathbb{E}}
\def\RR{\mathbb{R}}
\def\sr{\mathrm{SR1}}
\def\aaa{\mathrm{AAA}}
\def\algname{\texttt{AAA}}

\begin{document}
\title{Anderson Acceleration Without Restart: A Novel Method with $n$-Step Super Quadratic Convergence Rate}
\author{
Haishan Ye \thanks{Equal Contribution.} \thanks{School of Management; Xi'an Jiaotong University}
	\and
	Dachao Lin \footnotemark[1] \thanks{School of Mathematical Sciences; Peking University}
	\and
	Xiangyu Chang \footnotemark[2]	
	\and
	Zhihua Zhang \footnotemark[3]
}
\maketitle

\begin{abstract}
In this paper, we propose a novel Anderson's acceleration method to solve nonlinear equations, which does \emph{not} require a restart strategy to achieve numerical stability.  
We propose the greedy and random versions of our algorithm. Specifically, the greedy version selects the direction to maximize a certain measure of progress for approximating the current Jacobian matrix. 
In contrast, the random version chooses the random Gaussian vector as the direction to update the approximate Jacobian.
Furthermore, our algorithm, including both greedy and random versions, has an $n$-step super quadratic convergence rate, where $n$ is the dimension of the objective problem. 
For example, the explicit convergence rate of the random version can be presented as $ \norm{\vx_{k+n+1} - \vx_*} / \norm{\vx_k- \vx_*}^2 = \cO\left(\left(1-\frac{1}{n}\right)^{kn}\right)$ for any $k\geq 0$ where $\vx_*$ is the optimum of the objective problem.
This kind of convergence rate is new to Anderson's acceleration and quasi-Newton methods.
The experiments also validate the fast convergence rate of our algorithm.
\end{abstract}

\section{Introduction}

This paper focuses on solving a set of nonlinear equations:
\begin{equation}\label{eq:prob}
\mF(\vx) =\left(f_1(\vx), \dots, f_n(\vx)\right)^\top = \bm{0}_n,
\end{equation}
where $\mF:\sR^n\to \sR^n$ represents a differentiable vector-valued function with its Jacobian $\mJ(\vx)$ defined by $[\mJ(\vx)]_{ij} = \frac{\partial f_i(\vx)}{\partial \vx_j}$ for all $i, j \in [n]:=\{1,2,\dots,n\}$. 
The nonlinear equation problem presented in Eq.~\eqref{eq:prob} is prevalent across various fields, including natural and social sciences \citep{willert2014leveraging,briceno2013monotone}. 
For instance, solving a nonlinear equation of the Bellman operator in the infinite-horizon discounted Markov Decision Process reveals the optimal policy~\citep{bertsekas2012dynamic}. 
Additionally, identifying stationary points in optimization problems serves as another pertinent example~\citep{xiao2018regularized}, underscoring the wide-ranging applicability of solving such nonlinear equation systems.

Given the significance of the problem presented in Eq.~\eqref{eq:prob}, and considering the inherent challenges in obtaining analytical solutions, a wide array of numerical methods has emerged. 
Anderson's acceleration and quasi-Newton methods, known for their fast convergence rates, are classical solutions to this problem \cite{anderson1965iterative,broyden1965class}. 
Anderson's acceleration, in particular, boasts diverse applications across computational physics \cite{willert2014leveraging}, computational chemistry \cite{brezinski2018shanks}, and reinforcement learning \cite{shi2019regularized}, attesting to its versatility and efficacy. 
Similarly, quasi-Newton methods have been foundational in numerous applications, underscoring their importance in computational analyses \cite{xiao2018regularized,schaefer2015stabilized,liu2017quasi}. In response to these methods' critical roles, several variants have been developed to enhance convergence rates or ensure numerical stability, further expanding the scope and effectiveness of these computational techniques \cite{gay1978solving,hart1992solution,zhang2020globally,martinez2000practical}.

It is noteworthy that Anderson's acceleration and quasi-Newton methods share a close relationship. 
Anderson's acceleration can be considered a type of quasi-Newton method, specifically referred to as the generalized Broyden's ``bad'' (type-II) update~\citep{fang2009two,walker2011anderson}. 
Consequently, the original form of Anderson's acceleration is also termed \texttt{AA-II}. 
Viewing from Broyden's ``good'' (Type-I) update angle, a variant of Anderson's acceleration has been introduced, designated as type-I Anderson's acceleration (\texttt{AA-I})~\citep{zhang2020globally,fang2009two}. 
\citet{gay1978solving} demonstrated that \texttt{AA-I} is capable of achieving a superlinear convergence rate when equipped with full memory and a restart strategy. \citet{burdakov2005stable} proposed a marginally more stable version of full memory \texttt{AA-I}, generalizing \citet{gay1978solving}'s restart strategy. Employing a non-monotone line search method allows this method to be proven to converge globally \citep{burdakov2018multipoint}. More recently, \citet{zhang2020globally} have proposed a stable \texttt{AA-I} algorithm with the capability of global convergence. The convergence properties of \texttt{AA-II} have been extensively explored \citep{kelley2018numerical,toth2015convergence,walker2011anderson}. The quasi-Newton method is renowned for its ability to achieve a superlinear convergence rate \citep{nocedal2006numerical,kelley2003solving,broyden1973local}. Lately, a number of studies have elucidated the explicit convergence rates of quasi-Newton methods \citep{rodomanov2021greedy,rodomanov2021rates,jin2020non,lin2021faster,ye2021explicit}.

Despite decades of study and application, Anderson's acceleration and quasi-Newton methods continue to present unresolved challenges in the field \citep{fang2009two,kelley2018numerical,kelley2003solving}. 
One of the unsolvable issues is that \texttt{AA-I} suffer an inferior convergence rate due to the obligatory restart strategy \citep{gay1978solving}.
The restart causes that \texttt{AA-I} can not fully exploit the information lying in the previous $n$ search directions.
Thus, a novel \texttt{AA-I} without restart, which potentially can achieve an $n$-step quadratic convergence rate, is preferred in real applications.
Furthermore, recent studies have elucidated the explicit convergence rates of various quasi-Newton method variants \citep{jin2020non,rodomanov2021greedy,ye2021greedy}, and several works try to design novel quasi-Newton methods with faster convergence rate \citep{lin2022explicit,jin2022sharpened}.
The similar explicit convergence rates of the proposed novel \texttt{AA-I} should be determined. 
These open questions underscore the ongoing complexity and dynamism of research in this domain, highlighting the potential for significant advancements and discoveries.

This paper provides affirmative solutions to the previously mentioned problems. 
We introduce a novel type-I Anderson's acceleration method, termed Adjusted Anderson's Acceleration (\algname), that achieves numerical stability for \texttt{AA-I} without necessitating a restart strategy. 
Additionally, our algorithm \algname~is capable of attaining an $n$-step super quadratic convergence rate (refer to Definition~\ref{def:qq}). 
We further delineate an explicit super quadratic convergence rate of our \algname~for both greedy and random versions.
Such a convergence rate represents a novel advancement for both Anderson's acceleration and quasi-Newton methods. 
Finally, our experiments validate that our \algname~can achieve a super quadratic convergence rate.

\section{Preliminaries}\label{sec:preliminaries}
\subsection{Notation}
We denote vectors by lowercase bold letters (e.g., $ \vs, \vx$), and matrices by capital bold letters (e.g., $ \mW = [w_{i j}] $).
We use $[n]:=\{1,\dots,n\}$ and $\mI_n$ is the $\sR^{n \times n}$ identity matrix.
Moreover, $\norm{\cdot}$ denotes the $\ell_2$-norm (standard Euclidean norm) for vectors, or induced $2$-norm (spectral norm) for a given matrix: $\norm{\mA} = \sup_{\norm{\vu}=1, \vu\in\sR^n}\norm{\mA\vu}$, and $\norm{\cdot}_F$ denotes the Frobenius norm of a given matrix: $\norm{\mA}_F^2 = \sum_{i=1}^{m} \sum_{j=1}^n a_{i j}^2$, where $\mA=[a_{i j}] \in \sR^{m\times n}$.
We also adopt $\lambda_i(\mA)$ be its $ i $-th largest eigenvalue if $\mA$ is positive semi-definite. 
For two real symmetric matrices $\mA, \mB \in \sR^{n \times n}$, we denote $\mA\succeq\mB$ (or $\mB \preceq \mA$) if $\mA-\mB$ is a positive semidefinite matrix.
We use the standard $\cO(\cdot)$ notation to hide universal constant factors.
Finally, we recall the definition of the rate of $n$-step quadratic convergence used in this paper.
\begin{defn}[$n$-step super quadratic convergence]\label{def:qq}
	Suppose a sequence $\{\bm{x}_k\}$ converges to $\bm{x}_{\infty}$ that satisfies
	\[ \lim_{k\to \infty }\frac{\norm{\bm{x}_{k+n+1}-\bm{x}_{\infty}}}{\norm{\bm{x}_k-\bm{x}_{\infty}}^2} = q_k, \mbox{ with } q_k\to 0 \mbox{ as } k\to \infty.\]
	We say the sequence $\{\bm{x}_k\}$ converges with a $n$-step super quadratic convergence rate.
\end{defn}

\subsection{Two Assumptions}

To analyze the convergence properties of our novel algorithms,  we first list two standard assumptions in the convergence analysis of quasi-Newton methods.

\begin{assume}\label{ass:lisp}
	The Jacobian $\mJ(\vx)$ is Lipschitz continuous related to $\vx^*$ with
	parameter $M$ under induced $2$-norm, i.e., 
	\begin{equation*}
		\norm{\mJ(\vx)-\mJ(\vx^*)}_F \leq M \norm{ \vx-\vx^*}.
	\end{equation*}
\end{assume}

Such assumption also appears in \cite{jin2020non}, and the condition in
Assumption \ref{ass:lisp} only requires Lipschitz continuity in the direction of the optimal solution $\vx^*$. 
Indeed, that condition is weaker than general Lipschitz continuity for any two points. 
It is also weaker than the strong self-concordance recently proposed by \citet{rodomanov2021greedy, rodomanov2021new,rodomanov2021rates}.

In addition, the Jacobian at the optimum is nonsingular, shown in Assumption \ref{ass:optimal} below.

\begin{assume}\label{ass:optimal}
	The solution $\vx^*$ of Eq.~\eqref{eq:prob} is nondegenerate, i.e., $\mJ_{*}:=\mJ(\vx^*)$ is nonsingular or $\mJ_{*}^{-1}$ exists.
\end{assume}

Assumption \ref{ass:optimal} also appears in \cite{dennis1996numerical, nocedal2006numerical}.
Additionally, along with the classical result on quadratic convergence of Newton’s methods, we need some sort of regularity conditions for the objective Jacobian just as listed in  Assumption \ref{ass:lisp}.

\subsection{Anderson's Acceleration and Broyden's Method}

\paragraph{Anderson's Acceleration.} In the original Anderson's acceleration (\texttt{AA-II}) \citep{anderson1965iterative}, for each $k\ge 0$, one solves the the following least square problem with a normalization constraint:
\begin{equation} \label{eq:AA}
	\begin{aligned}
		\min_{\va} &\norm{\sum_{j=0}^{m_k} a_j \mF(\vx_{k-m_k+j})}\\
		\mbox{s.t.} &\sum_{j=0}^{m_k} a_j = 1,
	\end{aligned}
\end{equation}
where the integer $m_k$ is the memory in iteration $k$. 
Letting $\va_k = (a_{k, 0}, \dots, a_{k,m_k})^\top$ be the minimizer of Eq.~\eqref{eq:AA}, then the Anderson's acceleration updates $\vx_k$ as
\begin{equation*}
	\vx_{k+1} = \vx_k - \sum_{j=0}^{m_k} a_{k,j} \mF(\vx_{k-m_k+j}).
\end{equation*}
By denoting $\mS_k = [\vs_{k-m_k},\dots, \vs_{k-1}]$ and $\mY_k = [\vy_{k-m_k}, \dots, \vy_{k-1}]$ with $\vs_i = \vx_{i+1} - \vx_i$ and $\vy_i = \mF(\vx_{i+1}) - \mF(\vx_i)$, 
assuming that $\mY_k$ is of full column rank, then the update of \texttt{AA-II} can be represented as \citep{fang2009two}
\begin{align*}
	\vx_{k+1} 
	= \vx_k - \left(\mI + (\mS_k - \mY_k) \left(\mY_k^\top \mY_k\right)^{-1} \mY_k^\top\right)\mF(\vx_k)
	= \vx_k - \mH_k \mF(\vx_k).
\end{align*}
The above update can be regarded as a quasi-Newton-type update with $\mH_k$ being some sort of generalized  Broyden's ``bad'' (type-II) update \citep{broyden1965class,walker2011anderson}.
Thus, the original Anderson's acceleration is also called AA-II \citep{zhang2020globally}.

Type-I AA (AA-I) is a variant of original Anderson's acceleration derived from the perspective of Broyden's ``good'' (type-I) method \citep{fang2009two}. 
The update of \texttt{AA-I} is presented as follows with assuming that $\mS_k$ is of full column rank,
\begin{equation*}
	\vx_{k+1} = \vx_k - \left( \mI + (\mY_k - \mS_k)\left(\mS_k^\top\mS_k\right)^{-1} \mS_k^\top \right)^{-1}\mF(\vx_k)
	=\vx_k - \mB_k^{-1} \mF(\vx_k),
\end{equation*}
where 
\begin{equation*}
	\mB_k: = \mI + (\mY_k - \mS_k)\left(\mS_k^\top\mS_k\right)^{-1} \mS_k^\top.
\end{equation*}

\paragraph{Broyden's Method.}
Next, we will introduce the Broden's methods \citep{broyden1965class}. 
The so-called ``good'' Broyden's method is listed as follows:
\begin{equation*}
	\text{Broyden's ``good''}:
	\begin{cases}
	\vx_{k+1} =& \vx_k - \mB_k^{-1}\mF(\vx_k), \\
	\vy_k =& \mF(\vx_{k+1}) - \mF(\vx_k), \qquad \vs_k = \vx_{k+1} - \vx_k,\\
	\mB_{k+1} =& \mB_k + \frac{(\vy_k - \mB_k \vs_k) \vs_k^\top}{\vs_k^\top\vs_k}.
	\end{cases}
\end{equation*}
The Broyden's ``bad'' method is listed as follows:
\begin{equation*}
	\text{Broyden's ``bad''}:
	\begin{cases}
		\vx_{k+1} =& \vx_k - \mH_k\mF(\vx_k), \\
		\vy_k =& \mF(\vx_{k+1}) - \mF(\vx_k), \qquad \vs_k = \vx_{k+1} - \vx_k,\\
		\mH_{k+1} =& \mH_k + \frac{(\vs_k - \mH_k \vy_k) \vy_k^\top}{\vy_k^\top\vy_k}.
	\end{cases}
\end{equation*}

For a better understanding of the update rules of Broyden's method, we turn to a simple linear objective $\mF(\vx) =\mJ\vx-\vb$ for explanation. 
Then, both Broyden's ``good'' and ``bad'' methods have $\vy_k=\mJ\vs_k$.
Thus,
$$\mB_{k+1} = \mB_k + \frac{(\vy_k - \mB_k \vs_k) \vs_k^\top}{\vs_k^\top\vs_k}=\mB_k + \frac{(\mJ - \mB_k) \vs_k \vs_k^\top}{\vs_k^\top\vs_k},$$
and
$$\mH_{k+1} = \mH_k + \frac{(\vs_k - \mH_k \vy_k) \vy_k^\top}{\vy_k^\top\vy_k}=\mH_k + \frac{(\mJ^{-1} - \mH_k )\vy_k \vy_k^\top}{\vy_k^\top\vy_k}.
$$
We further denote that \textit{Broyden's Update} as:
\begin{defn}
	Letting $\mB, \mA\in\sR^{n \times n}$ and $\vs\in\sR^n$,  we define
	\begin{equation}\label{eq:broyden}
		 \text{Broyd} \left(\mB, \mA, \vs\right) := \mB +  \frac{\left(\mA-\mB\right)\vs\vs^\top}{\vs^\top\vs}.
	\end{equation}
\end{defn}
\noindent Then, for the simple linear objective, Broyden's ``good'' update can be written as $\mB_{k+1}=\text{Broyd} \left(\mB_k, \mJ, \vs_k\right)$, and the ``bad'' one is written as $\mH_{k+1}=\text{Broyd} \left(\mH_k, \mJ^{-1}, \vy_k\right)$. 
For the general non-linear equation, the update of $\mB_k$ in the above equation is a Broyden's update (Eq.~\eqref{eq:broyden}) with $\mA$ satisfying $\mA\vs_k = \vy_k$.

\section{Anderson Acceleration From Perspective of Greedy Quasi-Newton}

In this section, we will revisit the Anderson acceleration (type-I) from the perspective of the greedy quasi-Newton. 
Then we will show how the \texttt{AA-I} process approximates the Jacobian matrix gradually for the linear equation case where the Jacobian matrix is a constant matrix.

\subsection{Type-I Anderson Acceleration in Quasi-Newton Form}

We assume that the problem~\eqref{eq:prob} has a constant Jacobian $\mJ$.
Then the type-I Anderson acceleration has the following update rule in the quasi-Newton form \citep{fang2009two,zhang2020globally}
\begin{align}
	\vx_{k+1} =& \vx_k - \mB_k^{-1} \mF(\vx_k) \label{eq:x_up}\\
	\vs_k =& \vx_{k+1} - \vx_k\\
	\mB_{k+1} =& \mB_k + \frac{(\mJ - \mB_k)\vs_k (\mP^{\perp}_k \vs_k)^\top}{(\mP^{\perp}_k \vs_k)^\top \vs_k}, \label{eq:an_up} 
\end{align}
where $\mP^\perp_k = \mI - \mP_k$ with $\mP_k$ being the projection matrix of the range of $\vs_0, \dots, \vs_{k-1}$ and $\vs_0,\dots,\vs_k$ are linearly independent.
The update rule in Eq.~\eqref{eq:an_up} is very similar to the Broyden's type-I update in Eq.~\eqref{eq:broyden} but with the $\vs_k^\top$ replaced by $\left(\mP^\perp_k \vs_k\right)^\top $. 
Note that \texttt{AA-I} is first proposed in the work of \citet{gay1978solving}.

\subsection{Convergence Rate of AA-I for Linear Equation}

In this section, we will only consider the linear equation case with $\mF(\vx) = \mJ\vx - \vb$ whose Jacobian matrix is a constant matrix $\mJ$.
For the linear equation case, we will show that the matrix $\mB_k$ of \texttt{AA-I} will converge to the exact Jacobian $\mJ$ at most $n$ steps if all $\vs_k$'s are all independent.
First, we have the following proposition.
\begin{prop}\label{prop:aai}
Letting $\mJ\in\sR^{n\times n}$ be a Jacobian matrix and AA-I update as Eq.~\eqref{eq:x_up}-\eqref{eq:an_up} with $\vs_0,\dots,\vs_k$ being independent, then it holds that
\begin{align}
	&(\mB_{k+1} - \mJ) \vs_k = \bm{0}, \label{eq:null_space} \\
	&\mbox{rank}(\mB_k - \mJ) = \mbox{rank}(\mB_0 - \mJ) - k., \label{eq:rank_dec} \\
		&(\mJ - \mB_k) \mP_k^\perp = \mJ - \mB_k. \label{eq:perp}
\end{align}
\end{prop}
\begin{proof}
Note that, by the update rule of \texttt{AA-I}, it holds that 
\begin{align*}
	(\mB_{k+1} - \mJ) \vs_k \stackrel{\eqref{eq:an_up}}{=} (\mB_k - \mJ) \vs_k - (\mB_k - \mJ) \vs_k = \bm{0}. 
\end{align*}

Furthermore,  it holds that 
\begin{align}
	\mP_k^\perp \vs_k \perp \vs_0,\dots, \vs_{k-1}. \label{eq:s_perp}
\end{align}
This implies that for $i = 0,\dots, k-1$
\begin{align*}
	(\mB_{k+1} - \mJ) \vs_i = (\mB_k - \mJ) \vs_i + \frac{(\mJ - \mB_k)\vs_k \vs_k^\top \mP_k^\perp \vs_i}{(\mP^{\perp}_k \vs_k)^\top \vs_k} = (\mB_k - \mJ)\vs_i = \bm{0}.
\end{align*}
Thus, we can obtain that
\begin{equation*}
    \mbox{rank}(\mB_k - \mJ) = \mbox{rank}(\mB_0 - \mJ) - k.
\end{equation*}
Furthermore, we can also obtain that
\begin{align*}
	(\mJ - \mB_k) \mP_k^\perp = (\mJ - \mB_k) (\mI - \mP_k) = \mJ - \mB_k.
\end{align*}
\end{proof}

Based on the above proposition, we can obtain the following convergence property which describes how $\mB_k$ approaches the exact Jacobian $\mJ$.

\begin{thm}\label{thm:main_1}
Let $\{\vx_k\}$, $\{\vs_k\}$ and $\{\mB_k\}$ be generated by the \texttt{AA-I} (Eq.~\eqref{eq:x_up}-\eqref{eq:an_up}). 
If $\vs_k$ is independent of $\vs_0,\dots, \vs_{k-1}$, then it holds that
\begin{align}
	\norm{\mJ - \mB_{k+1}}_F^2 
	\le
	\left(1 - \frac{\lambda_{n - k} \left( (\mJ - \mB_k)^\top (\mJ - \mB_k) \right)}{\sum_{i = 1}^{n -k} \lambda_i \left( (\mJ - \mB_k)^\top (\mJ - \mB_k) \right)}\right) \cdot \norm{\mJ - \mB_k}_F^2.  \label{eq:B_dec}
\end{align} 
Letting $\vs_{k^*}$ be the first vector which dependents on $\vs_0,\dots,\vs_{k^*-1}$, then it holds that $\mF(\vx_{k^*+1}) = \bm{0}$. 
\end{thm} 
\begin{proof}
	First, we consider the case that $\vs_k$ is independent of $\vs_0,\dots, \vs_{k-1}$. 
	By the update rule of \texttt{AA-I}, we have
\begin{align*}
	&\norm{\mJ - \mB_{k+1}}_F^2 \\
	\stackrel{\eqref{eq:an_up}}{=}&
	\norm{\mB_k - \mJ + \frac{(\mJ - \mB_k)\vs_k (\mP^{\perp}_k \vs_k)^\top}{(\mP^{\perp}_k \vs_k)^\top \vs_k} }_F^2
	\\
	=&\norm{\mB_k - \mJ}_F^2 + \norm{\frac{(\mJ - \mB_k)\vs_k (\mP^{\perp}_k \vs_k)^\top}{(\mP^{\perp}_k \vs_k)^\top \vs_k}}_F^2 + \frac{2\tr\left((\mB_k - \mJ)^\top(\mJ - \mB_k)\vs_k (\mP^{\perp}_k \vs_k)^\top \right)}{\vs_k^\top \mP_k^\perp \vs_k}
	\\
	=&\norm{\mB_k - \mJ}_F^2
	+
	\frac{\tr\left(\vs_k^\top (\mJ - \mB_k)^\top(\mJ - \mB_k) \vs_k\right)}{\vs_k^\top \mP_k^\perp \vs_k}
	-
	 \frac{2\tr\left( \vs_k^\top \mP_k^\perp (\mJ - \mB_k)^\top(\mJ - \mB_k) \vs_k  \right)}{\vs_k^\top \mP_k^\perp \vs_k}
	\\
	\stackrel{\eqref{eq:perp}}{=}&
	\norm{\mB_k - \mJ}_F^2
	-
	\frac{\tr\left((\mP_k^\perp \vs_k)^\top (\mJ - \mB_k)^\top(\mJ - \mB_k) \mP_k^\perp \vs_k\right)}{\vs_k^\top \mP_k^\perp \vs_k}
	\\
	=&
	\norm{\mB_k - \mJ}_F^2
	-
	\frac{\norm{(\mJ - \mB_k) \mP_k^\perp \vs_k}^2}{\norm{\mP_k^\perp \vs_k}^2}.
\end{align*} 
Now we try to obtain the lower bound of $ \frac{\norm{(\mJ - \mB_k) \mP_k^\perp \vs_k}^2}{\norm{\mP_k^\perp \vs_k}^2} $.
Since  $\vs_0,\dots,\vs_k$ are independent, by Eq.~\eqref{eq:rank_dec}, we can obtain that
\begin{align*}
	\mathrm{rank}(\mJ - \mB_k) \leq n-k.
\end{align*}

Furthermore, by Eq.~\eqref{eq:null_space} and \eqref{eq:s_perp}, we can obtain that 
\begin{align*}
	\frac{\norm{(\mJ - \mB_k) \mP_k^\perp \vs_k}^2}{\norm{ \mP_k^\perp \vs_k}^2}
	\ge
	\lambda_{n - k} \left( (\mJ - \mB_k)^\top (\mJ - \mB_k) \right).
\end{align*}
It also holds that
\begin{align*}
	\norm{\mJ - \mB_k}_F^2 = \sum_{i = 1}^{n -k} \lambda_i \left( (\mJ - \mB_k)^\top (\mJ - \mB_k) \right).
\end{align*}
Therefore, we can obtain that
\begin{align*}
	\norm{\mJ - \mB_{k+1}}_F^2 
	\le
	\left(1 - \frac{\lambda_{n - k} \left( (\mJ - \mB_k)^\top (\mJ - \mB_k) \right)}{\sum_{i = 1}^{n -k} \lambda_i \left( (\mJ - \mB_k)^\top (\mJ - \mB_k) \right)}\right) \cdot \norm{\mJ - \mB_k}_F^2. 
\end{align*}

Next, we will consider the case that $\vs_{k^*}$ be the first vector which is dependent of $\vs_0,\dots, \vs_{k^*-1}$.
By Eq.~\eqref{eq:null_space}, we can obtain that for $i=0,\dots, k^*-1$, it holds that $(\mB_k - \mJ)\vs_i = \bm{0}$.
Since, $\vs_{k^*}$ is dependent on $\vs_0,\dots, \vs_{k^*-1}$. 
Thus, we have $(\mB_{k^*} - \mJ)\vs_{k^*} = \bm{0}$ which implies that $\mB_{k^*}\vs_{k^*} = \mJ\vs_{k^*}$. 
By the update rule of \texttt{AA-I} in Eq.~\eqref{eq:x_up}, we can obtain that $\mB_{k^*}\vs_{k^*} = -\mF(\vx_{k^*})$. 
At the same time, it holds that $\mJ\vs_{k^*} = \mJ(\vx_{k^*+1} - \vx_{k^*}) = \mF(\vx_{k^*+1}) -\mF(\vx_{k^*}) $.
Combining with   $\mB_{k^*}\vs_{k^*} = \mJ\vs_{k^*}$, we can obtain that $\mF(\vx_{k^*+1}) = \bm{0}$ which implies that $\vx_{k^*+1}$ is a solution of $\mF(\vx) = \bm{0}$.
\end{proof}

\begin{remark}
	Theorem~\ref{thm:main_1} is almost the same to Theorem~3.1 of \citet{gay1978solving} but we provide an explicit convergence rate of $\norm{\mB_k - \mJ}_F$ in the update of \texttt{AA-I}.
Eq.~\eqref{eq:B_dec} shows that $\mB_k$ will gradually converge to the exact Jacobian $\mJ$. This property is  the same to the one of greedy quasi-Newton \cite{rodomanov2021greedy,lin2021faster,ye2021greedy}.
Thus, we can regard Anderson acceleration as a kind of greedy quasi-Newton method.
Furthermore, $\mB_k$ will converge to $\mJ$ at most $n$ steps since when $k=n-1$, it holds that
\begin{align*}
	1 - \frac{\lambda_{n - k} \left( (\mJ - \mB_k)^\top (\mJ - \mB_k) \right)}{\sum_{i = 1}^{n -k} \lambda_i \left( (\mJ - \mB_k)^\top (\mJ - \mB_k) \right)} 
	= 
	1 - \frac{\lambda_1 \left((\mJ - \mB_{n-1})^\top (\mJ - \mB_{n-1})\right) }{\lambda_1 \left((\mJ - \mB_{n-1})^\top (\mJ - \mB_{n-1})\right)} =0.
\end{align*}
\end{remark}
\begin{remark}
The property that the matrix $\mB_k$ of \texttt{AA-I} converges to the exact Jacobian at most $n$ steps mainly bases properties shown in Eq.~\eqref{eq:null_space}-\eqref{eq:perp}. 
Similar properties also hold for the greedy SR1 algorithm (Refer to Theorem 3.5 of \cite{rodomanov2021greedy}), which leads that the greedy SR1 algorithm can find the optimum for the quadratic optimization problem. 
These properties are because of the update of \texttt{AA-I} exploits the special property of $\mP_k^\perp \vs_k$.  
Unfortunately, when $k = n$, $\mP_k^\perp \vs_k$ will be a zeroth vector and the update in Eq.~\eqref{eq:an_up} collapses.
This is the reason why solving the general non-linear systems has to restart the update of \texttt{AA-I} even it takes a full memory, that is, storing $\vs_{k}, \vs_{k-1},\dots, \vs_{k-n}$ \citep{gay1978solving}.
\end{remark}
\begin{remark}\label{rmk:quad}
	The famous conjugate gradient method can achieve a $n$-step quadratic convergence rate with a restart strategy \citep{cohen1972rate}, that is, restart for each $n$ step. 
	This property relies on the fact that the conjugate gradient method can find the optimal point of a quadratic function at most $n$ steps.
	However, the problem in practice is commonly of high dimension, i.e., $n$ is large.
	Thus, the conjugate gradient with the  $n$-step restart strategy are not relevant in a practical context \cite{nocedal2006numerical}.
    \texttt{AA-I} with a restart strategy can achieve a superlinear convergence rate \citep{gay1978solving}. 
    However, due to the restarting, \texttt{AA-I} can not full exploit the information in the previous $n$ directions which will lead to an inferior convergence rate compared with our method. 
	In this paper, we try to propose a novel adjusted Anderson acceleration method which can  achieve a $n$-step super quadratic convergence rate but without a restart.
\end{remark}

\section{ Anderson Acceleration Without Restart}
\label{sec:aaa}

In this section, we will propose our novel Anderson's acceleration method, which does \emph{not} require a restart strategy to achieve the numerical stability.
First, we give the algorithm description and list the main procedure of our algorithm.
Then, we will show that our algorithm can find the optimum at most $n$ steps for the linear equation problem, which is the same as the \texttt{AA-I}.
Finally, we give a convergence analysis of our algorithm for the general non-linear equation problems and provide explicit super quadratic convergence rates of our algorithm, including the greedy and random versions. 
\subsection{Algorithm Description}

In order to design a novel quasi-Newton whose approximate Jacobian $\mB_k$ satisfies the properties in Proposition~\ref{prop:aai} but without a restart strategy, we resort to the update rule of the SR1 algorithm.
Accordingly, we propose a new algorithm named Adjusted Anderson Acceleration (\algname), whose update combines the \texttt{AA-I} update with the one of SR1. 
Given matrices $\mB$, $\mJ$ and a vector $\vs$, our \algname~takes the following update rule:
\begin{equation}\label{eq:aaa_up}
	\aaa(\mB, \mJ, \vs) = 
	\begin{cases}
		\mB + \frac{(\mJ-\mB)\vs\vs^\top (\mJ-\mB)^\top}{\vs^\top (\mJ - \mB)^\top(\mJ - \mB)\vs} (\mJ - \mB), \qquad &\mbox{if } (\mJ - \mB)\vs \neq \bm{0}, \\
		\mB, & \mbox{otherwise}.
	\end{cases}
\end{equation}
For comparison, the SR1 update is defined as follows for given two positive definite matrices $\mH$, $\mG$ and a vector $\vu$, 
\begin{equation}\label{eq:sr1}
	\sr(\mG,\mH,\vu) = \begin{cases}
		\mG - \frac{(\mG - \mH)\vu\vu^\top(\mG - \mH)}{\vu^\top(\mG - \mH)\vu },  \qquad & \mbox{if } (\mG - \mH)\vu\neq \bm{0},\\
		\mG, & \mbox{otherwise}.
	\end{cases}
\end{equation}

Comparing Eq.~\eqref{eq:aaa_up} with Eq.~\eqref{eq:sr1}, we can observe that they are similar to each other. 
\algname~only has an extra $\mB-\mJ$ term compared with the SR1 update if $\mB = \mG$ and $\mJ = \mH$.
Due to this similarity, our \algname~shares some advantages of  the SR1 method such as \algname~does not require a restart strategy.
However, these two updates have some differences. 
For example, in the update of Eq.~\eqref{eq:aaa_up}, the matrices $\mB$ and $\mJ$ are commonly asymmetric while the SR1 update commonly requires $\mG$ and $\mH$ to be positive definite.

Comparing Eq.~\eqref{eq:aaa_up} with the update of \texttt{AA-I} in Eq.~\eqref{eq:an_up}, we can observe that our \algname~replaces $\mP_k^\perp$ with $\mJ - \mB$. 
This is because it holds that $\mP_k^\perp = \mP_k^\perp \mP_k^\perp$ and $\mP_k^\perp = \left[\mP_k^\perp \right]^\top$. 
And accordingly,  Eq.~\eqref{eq:an_up} can be represented as:
\begin{equation*}
\mB_{k+1} = \mB_k + \frac{(\mJ - \mB_k)\vs_k \vs_k^\top \left[\mP_k^\perp\right]^\top }{ \vs_k^\top \left[\mP_k^\perp\right]^\top \mP^{\perp}_k\vs_k} \cdot \mP_k^\perp.
\end{equation*} 

Thus, our \algname~can be regarded as a kind of Anderson acceleration method but with some modifications coming from the SR1 update.

Next, we will propose two methods to choose $\vs$.
The first method is choose $\vs$ greedily such that 
\begin{equation}\label{eq:greedy}
	\vs(\mB,\mJ) = \argmax_{\vs\in\{\ve_1,\dots,\ve_n\}} \norm{(\mJ - \mB)\vs}.
\end{equation} 
Accordingly, we call this method as the \emph{greedy} \algname.
The second method uses a random vector $\vs \sim \cN(0, \mI_n)$. Accordingly, we call this method as the \emph{random} \algname.

With the update rule in Eq.~\eqref{eq:aaa_up} and two methods for choosing $\vs$, we can obtain the following algorithmic procedure of \algname~for solving problem~\eqref{eq:prob}:
\begin{equation} \label{eq:aaa}
	\begin{cases}
		& \vx_{k+1} = \vx_k-\mB_k^{-1} \mF(\vx_k),   \\
		& \mbox{Choose } \vs_k \mbox{ greedily or randomly}, \\ 
	& \mB_{k+1} = \aaa(\mB_k, \mJ(\vx_k), \vs_k) .
	\end{cases}
\end{equation}
The detailed algorithm is listed in Algorithm~\ref{algo:aaa-greedy-random}.
Note that Algorithm~\ref{algo:aaa-greedy-random} is mainly for the convergence analysis.
To obtain the computation efficiency, the inverse update of \algname~should be conducted as follows:
\begin{equation*}
	\begin{cases}
			& \vx_{k+1} = \vx_k-\mC_k \mF(\vx_k),    \\
			& \mC_{k+1} = \mC_k + \frac{\mC_k \left(\mR_k\vs_k\right) \cdot \left(\vs_k^\top \mR_k^\top\mR_k\mC_k\right)}{\vs_k^\top\mR_k^\top\mR_k\vs_k+\left(\vs_k^\top\mR_k^\top\mR_k\mC_k\right)\left(\mR_k\vs_k\right)}, \quad \mR_k := \mB_k-\mJ_k.  
	\end{cases}
\end{equation*} 

\begin{remark}
In the real implementation of \texttt{AA-I}, the $\mJ_k\vs_k$ is commonly replaced by $\mF(\vx_{k+1}) - \mF(\vx_k)$. 
However, in the update of \algname, it requires to compute $\mJ_{k}\vs_k$. 
Fortunately, the computation of $\mJ_k\vs_k$ can be obtained efficiently and its computation cost is almost the same to evaluating $\mF(\vx)$ (refer to Chapter 8.2 of \citet{nocedal2006numerical}).
In the implementation of the greedy \algname, one has to solve the problem~\eqref{eq:greedy} which seems requiring to compute the $\mJ-\mB$ explicitly. 
We can use the JL Lemma to reduce the computation \cite{lindenstrauss1984extensions}.
We use a Gaussian random matrix $\mU\in\RR^{m\times n}$ with $m = \cO(\log n)$ and compute $\mU(\mJ - \mB)$. 
Note that $\mU(\mJ - \mB)$ can also be computed efficiently at the cost about $\cO(\log n)$ times of computing $\mF(\vx)$ (refer to Chapter 8.2 of \citet{nocedal2006numerical}). 
Then, we will find a vector $\vs'$ such that maximizes $\norm{\mU(\mJ - \mB)\vs}$ which can be computed very efficiently.
The property of the JL Lemma shows that it holds with a high probability that $\frac{1}{2} \norm{ (\mJ - \mB)_{:,i} } \leq \norm{\mU((\mJ - \mB)_{:,i})} \leq \frac{3}{2} \norm{ (\mJ - \mB)_{:,i} }$ \cite{kane2014sparser}.
Thus, it holds that $\norm{(\mJ -\mB)\vs'} \geq \frac{1}{6}\max_{\vs\in\{\ve_1,\dots,\ve_n\}} \norm{(\mJ - \mB)\vs}$.
That is, the greedy direction $\vs'$ computed with the JL Lemma approximation is also a good direction. 
\end{remark}

\begin{algorithm}[t]
	\caption{Greedy or Random \algname~Method.}
	\begin{algorithmic}[1]
		\STATE Initialization: set $\mB_0, \vx_0$.
		\FOR{$ k \geq 0 $}
		\STATE Update $\bm{x}_{k+1} = \bm{x}_{k} - \mB_k^{-1} \mF(\vx_k)$.
		\STATE Choose $\vs_k$ from \\
		1) \textit{greedy update}: $ \vs_k = \mathop{\arg\max}_{\vs \in \{\ve_1, \dots, \ve_n\}} \norm{(\mB_k-\mJ_k)\vs}$ with $\mJ_k$ being the Jacobian at $\vx_k$, or \\
		2) \textit{random update}: $\vs_k \sim \cN(0, \mI_n)$.
		\STATE Compute $\mB_{k+1} = \aaa(\mB_k, \mJ_k, \vs_k)$ by Eq.~\eqref{eq:aaa_up}.
		\ENDFOR
	\end{algorithmic}
	\label{algo:aaa-greedy-random}
\end{algorithm}

\subsection{  $n$-Step Convergence for Linear Equation}
In this section, we consider the linear equation case $\mF(\vx) = \mJ\vx - \vb$ with $\mJ$ being invertible.
Let us denote $\mR_k = \mB_k - \mJ$. 
The the update in Eq.~\eqref{eq:aaa_up} can be represented as  
\begin{align} \label{eq:R_up}
	\mR_{k+1} = 
	\begin{cases}
		\mR_k - \frac{\mR_k \vs_k \vs_k^\top \mR_k^\top}{\vs_k^\top \mR_k^\top \mR_k \vs_k}\mR_k, \qquad &\mbox{if} \quad	( \mJ-\mB_k )\vs_k \neq 0,\\
		\mR_k, &\mbox{otherwise}.
	\end{cases}
\end{align}

\begin{prop}\label{prop:R}
Given $\vs_0\dots\vs_k$, and $\mB_k$ is generated by Algorithm~\ref{algo:aaa-greedy-random} with a constant Jacobian matrix $\mJ$, then it holds that 
\begin{align}
	&(\mB_{k+1} - \mJ)\vs_k = \bm{0}, \label{eq:null_space_1}\\
		&\ker(\mB_k - \mJ) \in \ker(\mB_{k+1} - \mJ). \label{eq:rank_dec_1}
\end{align}
\end{prop}
\begin{proof}
	If $\mR_k\vs_k \neq \bm{0}$, then by Eq.~\eqref{eq:R_up}, it holds that
\begin{align*}
	\mR_{k+1} \vs_k 
	= 
	\mR_k\vs_k - \frac{\mR_k \vs_k \vs_k^\top \mR_k^\top}{\vs_k^\top \mR_k^\top \mR_k \vs_k}\mR_k \vs_k
	=
	\mR_k\vs_k - \mR_k\vs_k 
	=\bm{0}.
\end{align*}
If $\mR_k\vs_k = \bm{0} $, then by Eq.~\eqref{eq:R_up}, it holds that
\begin{align*}
	\mR_{k+1} \vs_k = \mR_k\vs_k = \bm{0}.
\end{align*}

Furthermore, for $\vs_i$ with $i=0,\dots, k-1$,
\begin{align*}
	\mR_{k+1}\vs_i 
	= 
	\mR\vs_i - \frac{\mR_k \vs_k \vs_k^\top \mR_k^\top}{\vs_k^\top \mR_k^\top \mR_k \vs_k}\mR_k \vs_i 
	=
	\left(\mI -\frac{\mR_k \vs_k \vs_k^\top \mR_k^\top}{\vs_k^\top \mR_k^\top \mR_k \vs_k} \right) \mR\vs_i.
\end{align*}

Thus, we can obtain that
\begin{align*}
	&\mathrm{Ker}(\mR_{k}) \subseteq \mathrm{Ker}(\mR_{k+1}),\\
	&\vs_k \in \mathrm{Ker}(\mR_{k+1}).
\end{align*}
\end{proof}

\begin{thm}\label{thm:main2}
	Given an initial non-singular matrix $\mB_0 \in \RR^{n\times n}$, update $\mB_k$ as Eq.~\eqref{eq:aaa_up} with $\vs_k$ chosen greedily or independently randomly chosen from $\cN(0, \mI_n)$.
	Then our algorithm finds the optimum at most $n$ steps. Letting $k_*$ be the first index such that $ (\mB_{k_*} - \mJ)\vs_{k_*} = \bm{0}$, then it holds that $\mB_{k_*} = \mJ$. For $k < k_*$, the approximate Jacobian $\mB_k$ satisfies that
	\begin{equation}
		\EE\left[ \norm{\mB_k - \mJ}_F^2 \right] 
		\leq
		\left(1 - \frac{k}{n}\right)_+ \norm{\mB_0 - J}_F^2, \label{eq:sup_decay}
	\end{equation}
	where $(x)_+ = \max\{0, x\}$. 
\end{thm}
\begin{proof}
	Letting $k_* \leq n$ be the first index that $\mR_{k_*} \vs_{k_*} = \bm{0}$, 
for the greedy method, we have
\[ \vs_{k_*} = \mathop{\arg\max}_{\vs \in \{\ve_1, \dots, \ve_n\}} \norm{\mR_{k_*}\vs}. \]
We also have
\[  0 \leq \norm{\mR_{k_*}}_F^2 = \sum_{i=1}^n \norm{\mR_{k_*}\ve_i}^2 \leq n \norm{\mR_{k_*}\vs_{k_*}}^2 = 0,   \]
which leads to $\mR_{k_*}=\bm{0}$, that is $\mJ = \mR_{k_*}$.
For the random method,  $\vs_{k_*}$ is a random Gaussian vector. 
Letting $\mB_{k_*} - \mJ = \mU \mathbf{\Lambda} \mU^\top$ be the spectral decomposition with $\mathbf{\Lambda} = \diag(\lambda_1,\dots, \lambda_n)$, then we have 
	\begin{align}
		\norm{ (\mB_{k_*} - \mJ)  \vs_{k_*} }^2 = \norm{ \mU \mathbf{\Lambda} \mU^\top \vs_{k_*} }^2 = \norm{\mathbf{\Lambda} \mU^\top \vs_{k_*}}^2 = \sum_{i=1}^{n}\lambda_i^2 u_i^2 = 0, \label{eq:zero}
	\end{align}
	where $\vu = \mU^\top\vs_{k_*} =  [u_1, \dots, u_n]$ is also a Gaussian random vector.
	By the Eq.~\eqref{eq:zero}, we can obtain that $\Lambda = \diag(\bm{0})$ holds with probability $1$, which implies  that $\mB_{k_*} = \mJ$.
	Thus, for both the greedy and random version of \algname, if   $\mR_{k_*} \vs_{k_*} = \bm{0}$, then it holds that $\mB_{k_*} = \mJ$.
	Then by the update rule of our algorithm, it holds that $\mF(\vx_{k^*+1}) = \bm{0}$.
	
	For all $k < k_*$, we have
\begin{align*}
	\norm{\mR_{k+1}}_F^2 
	\stackrel{\eqref{eq:R_up}}{=}&
	\norm{\mR_k - 	 \frac{\mR_k \vs_k \vs_k^\top \mR_k^\top}{\vs_k^\top \mR_k^\top \mR_k \vs_k}\mR_k }_F^2\\
	=& 
	\norm{\mR_k}_F^2 + \frac{\tr\left( \mR_k \vs_k \vs_k^\top   (\mR_k^\top \mR_k)^2 \vs_k\vs_k^\top \mR_k^\top \right)}{(\vs_k^\top \mR_k^\top \mR_k \vs_k)^2} 
	- \frac{2\tr\left( \mR_k^\top \mR_k\vs_k\vs_k^\top \mR_k^\top \mR_k \right)}{\vs_k^\top \mR_k^\top \mR_k \vs_k}\\
	=&
	\norm{\mR_k}_F^2 + \frac{\vs_k^\top   (\mR_k^\top \mR_k)^2 \vs_k}{\vs_k^\top \mR_k^\top \mR_k \vs_k} - 2\frac{\vs_k^\top   (\mR_k^\top \mR_k)^2 \vs_k}{\vs_k^\top \mR_k^\top \mR_k \vs_k}\\
	=&
	\norm{\mR_k}_F^2 - \frac{\vs_k^\top   (\mR_k^\top \mR_k)^2 \vs_k}{\vs_k^\top \mR_k^\top \mR_k \vs_k}\\
	\le&
	\norm{\mR_k}_F^2 - \frac{\vs_k^\top   (\mR_k^\top \mR_k) \vs_k}{\vs_k^\top  \vs_k},
\end{align*}
where the last inequality is because of the Cauchy's inequality, that is,
\begin{align*}
	\dotprod{\vs_k,  (\mR_k^\top \mR_k) \vs_k}^2 \le \dotprod{\vs_k, \vs_k} \cdot \dotprod{ (\mR_k^\top \mR_k) \vs_k, (\mR_k^\top \mR_k) \vs_k }.
\end{align*}

{\bf For the greedy method}, at step $k$, since $\mR_k \vs_i = \bm{0}$ for $i<k$ from Proposition~\ref{prop:R}.
Thus, at step $k$, $\vs_k$ only as $n-k$ possible choices and rank$(\mR_k)\leq n-k$. 
Without loss of generality, we denote these possible choices as $\{\ve_1,\dots,\ve_{n-k}\}$. 
Then it holds that $\mR_k\ve_i = \bm{0}$ for all $i\ge n-k+1$. 
Accordingly, it holds that $(\mR_k^\top \mR_k)_{i,i} = 0$ for $i\ge n-k+1$.
Thus, we can obtain that
\begin{align*}
\frac{\vs_k^\top   (\mR_k^\top \mR_k) \vs_k}{\vs_k^\top  \vs_k} 
= 
\max_{\vs} \frac{\vs^\top(\mR_k^\top \mR_k)\vs}{\vs^\top \vs}
=
\max_{i\in[n-k]} (\mR_k^\top \mR_k)_{i,i}
\geq
\frac{1}{n-k}\tr\left( \mR_k^\top \mR_k \right)
=
\frac{1}{n-k} \norm{\mR_k}_F^2.
\end{align*}
Therefore, the greedy choice of $\vs_k$ leads to 
\begin{align*}
\norm{\mR_{k+1}}_F^2  
\leq \left(1 - \frac{1}{n-k}\right) \norm{\mR_k}_F^2.
\end{align*}
Consequently,
\begin{align*}
\norm{\mR_k}_F^2  
\leq \prod_{i=0}^{k-1} \left(1 - \frac{1}{n-i}\right) \norm{\mR_0}_F^2 
= \left(1 - \frac{k}{n}\right) \norm{\mR_0}_F^2.
\end{align*}

{\bf For the random method}, for all $k\geq 0$, $\vs_k$'s are independently chosen from $\cN(0, \mI_n)$.
Suppose $r_k:=\mathrm{rank}(\mR_k) \ge 1$, i.e., $\mR_k \neq \bm{0}$.
We denote $\mR_k^\top \mR_k = \mU_k \mathbf{\Lambda}_k\mU_k^\top$ as the spectral decomposition of $\mR_k^\top \mR_k$ with an orthonormal matrix $\mU_k$ and a diagonal matrix $ \mathbf{\Lambda}_k = \diag(\lambda_1,\dots,\lambda_{r_k},0,\dots,0) $.
Denoting that $\vv_k = (v_1,\dots,v_d)^\top = \mU_k^\top \vs_k$, then we can obtain that
\begin{align*}
&\EE_{\vs_k} \left[\frac{\vs_k^\top \mR_k^\top\mR_k \vs_k}{\vs_k^\top \vs_k}\right]
= 
\EE_{\vs_k} \left[ \frac{\sum_{i=1}^{r_k} \lambda_i v_i^2}{\sum_{i = 1}^{r_k}v_i^2} \right]
=
\sum_{i=1}^{r_k} \lambda_i\EE_{\vv_k} \left[\frac{v_i^2}{\sum_{i=1}^{r_k} v_i^2}\right]\\
=&
\frac{1}{r_k} \sum_{i=1}^{r_k} \lambda_i
= 
\frac{\tr(\mR_k^\top \mR_k)}{r_k}
=
\frac{\norm{\mR_k}_F^2}{r_k}.
\end{align*}
Therefore,  $\vs_k\sim \cN(0, \mI_n)$ leads to
\begin{align*}
	\EE_{\vs_k} \left[\norm{\mR_{k+1}}_F^2\right] 
	\leq 
	\left(1 - \frac{1}{n-k}\right) \norm{\mR_k}_F^2.
\end{align*} 
Consequently, by the law of total expectation, we can obtain that
\begin{align*}
	\EE\left[\norm{\mR_k}_F^2\right]  
	\leq \prod_{i=0}^{k-1} \left(1 - \frac{1}{n-i}\right) \norm{\mR_0}_F^2 
	= \left(1 - \frac{k}{n}\right) \norm{\mR_0}_F^2.
\end{align*}



\end{proof}

\begin{remark}
Comparing Eq.~\eqref{eq:sup_decay} with Eq.~\eqref{eq:B_dec}, we can conclude that both AA-I and our algorithm whose $\mB_k$ will converge to the exact Jacobian at most $n$ steps.
However, AA-I method exploits the special property of $\mP_k^\perp \vs_k$ which will be a zero vector at the $(n+1)$-th step and the update in Eq.~\eqref{eq:an_up} collapses.
Thus, one has to restart the update of \texttt{AA-I}.
In contrast, our method does not have such weakness. 
For the general case that the Jacobian is not a constant matrix, if $\mB_k - \mJ_k \neq \bm{0}$ with $\mJ_k$ being the Jocobian matrix at $k$-th iteration, then our algorithm can still work without restart even when $k \geq n$.  
\end{remark}

\subsection{$n$-step Super Quadratic Convergence for General Non-linear Equation}

For the notation brevity, we introduce the following notations: 
\begin{equation*}
\sigma_k := \norm{\mB_k-\mJ_*}_F, r_k := \norm{\vx_k-\vx_*}, c := \norm{\mJ_*^{-1}}, \mJ_* := \mJ(\vx_*), \mJ_k := \mJ(\vx_k), \mR_k:= \mB_k -\mJ_k.
\end{equation*}

First,  the sequences $\{\mB_k\}$ and $\{\vs_k\}$ generated by our algorithm satisfy the following proposition.
\begin{prop}
	Letting $\mB_k$ and $\vs_k$ be generated by Algorithm~\ref{algo:aaa-greedy-random}, then it holds that
	\begin{equation}
		(\mB_{k+1} -\mJ_*) \vs_k = (\mJ_k - \mJ_*)\vs_k. \label{eq:null}
	\end{equation}
\end{prop}
\begin{proof}
By the update rule in Algorithm~\ref{algo:aaa-greedy-random}, and assuming that $ \mR_k \vs_k \neq \bm{0} $,  we have
\begin{align*}
	(\mB_{k+1} -\mJ_*) \vs_k 
	= &
	(\mB_{k+1} - \mJ_k)\vs_k + (\mJ_k - \mJ_*)\vs_k  \\
	\stackrel{\eqref{eq:aaa_up}}{=}&
	\mR_k \vs_k - \frac{\mR_k \vs_k\vs_k^\top \mR_k^\top}{\vs_k^\top \mR_k^\top\mR_k \vs_k} \mR_k\vs_k  + (\mJ_k - \mJ_*)\vs_k\\
	=& 
	(\mJ_k - \mJ_*)\vs_k.
\end{align*}
If $ \mR_k \vs_k = \bm{0} $, then it holds that
\begin{align*}
(\mB_{k+1} -\mJ_*) \vs_k 
= 
(\mB_{k+1} - \mJ_k)\vs_k + (\mJ_k - \mJ_*)\vs_k 
\stackrel{\eqref{eq:aaa_up}}{=}
(\mB_k - \mJ_k)\vs_k + (\mJ_k - \mJ_*)\vs_k 
=(\mJ_k - \mJ_*)\vs_k.
\end{align*}
\end{proof}

\begin{remark}
Eq.~\eqref{eq:null} shows that $\vs_k$ almost lies in the null space of $\mB_{k+1} - \mJ_*$. This is the key to the $n$-step super quadratic convergence of our algorithm. 
Eq.~\eqref{eq:null} reduce to Eq.~\eqref{eq:null_space_1} as  $\mJ_k$ goes to $\mJ_*$.
However, because of  $\vs_k$ does not lies in the null space of $\mB_{k+1} - \mJ_*$, Eq.~\eqref{eq:rank_dec_1} can not hold.
Thus, our \algname~can not obtain an exact approximation of $\mJ_*$ at most $n$ steps for the general nonlinear equation.
\end{remark}

In the following lemmas, we will show how $r_k$ and $\sigma_k$ correlate with each other and these lemmas help to obtain the explicit super quadratic convergence rates of our algorithm including both the greedy and random versions.

\begin{lemma}[Lemma 5.3 of \citet{ye2021greedy}]
Let the objective function satisfy Assumption~\ref{ass:lisp} and \ref{ass:optimal}. Then $r_k$ and $\sigma_k$ generated by Algorithm~\ref{algo:aaa-greedy-random}  have the following property:
\begin{equation}\label{eq:rr}
    r_{k+1} \leq \frac{3cMr_k/2 + c\sigma_k}{1 - c (\sigma_k +Mr_k)} \cdot r_k, \mbox{ if } c(\sigma_k + Mr_k) < 1.
\end{equation}
\end{lemma}

\begin{lemma}
	Supposing the initialization of Algorithm~\ref{algo:aaa-greedy-random} satisfies
	\begin{equation}\label{eq:init}
		c(\sigma_0 + 4Mr_0) \leq 1/3, 
	\end{equation}
	then it holds  that
	\begin{equation} \label{eq:rr_dec}
		r_{k+1} \leq r_k / 2.
	\end{equation}
\end{lemma}
\begin{proof}
	We prove the result by induction. For $k=1$, we have that
	\begin{align*}
		r_1\stackrel{\eqref{eq:rr}}{\leq}
		\frac{3cMr_0/2 + c\sigma_0}{1 - c (\sigma_0 +Mr_0)}r_0  
		\stackrel{\eqref{eq:init}}{\leq} 
		\frac{1/3}{1 - 1/3} r_0 =
		\frac{r_0}{2}.
	\end{align*}
	Supposing that Eq.~\eqref{eq:rr_dec} holds for $k$,
	then we have that
	\begin{align*}
		r_{k+1}
		\stackrel{\eqref{eq:rr}}{\le}&
		\frac{3cMr_k/2 + c\sigma_k}{1 - c (\sigma_k +Mr_k)} r_k
		\stackrel{\eqref{eq:ss_dec}}{\leq}
		\frac{3cMr_k/2 + c(\sigma_0 + M\sum_{t=0}^{k}r_t)}{1 - c\left(\sigma_0 + M\sum_{t=0}^{k}r_t + Mr_k\right)} r_k\\
		\leq&
		\frac{ c\left( \sigma_0 + 2Mr_0 + 3M r_0 2^{-(k+1)}\right)}{1 - c\left(\sigma_0 + 2Mr_0 + Mr_02^{-k}\right)}r_k
		\leq
		\frac{ c\left( \sigma_0 + 4Mr_0 \right)}{1 - c\left(\sigma_0 + 4Mr_0\right)} r_k
		\stackrel{\eqref{eq:init}}{\leq}
		r_k /2,
	\end{align*}
	where the third inequality is because of the inductive hypothesis that $r_i \leq 2^{-i}r_0$, for all $i\leq k$.
	Thus, the result in Eq.~\eqref{eq:rr_dec} has been proved.
\end{proof}

\begin{lemma}
The sequences $\{\sigma_k\}$ and $\{r_k\}$ generated by Algorithm~\ref{algo:aaa-greedy-random} for both greedy and random versions satisfy the following property:
	\begin{equation}
		\EE\left[\sigma_{k+1}\right] \leq \left(1 -\frac{1}{2n} \right) \sigma_k + \left(2-\frac{1}{2n}\right)Mr_k. \label{eq:ssig}
	\end{equation}
\end{lemma}
\begin{proof}
First, we consider the case that $\mR_k \vs_k \neq \bm{0}$. 
By the update rule in Algorithm~\ref{algo:aaa-greedy-random} and the definition of $\mR_k$, we have
\begin{align*}
    \norm{\mB_{k+1} - \mJ_k}_F^2 
    \stackrel{\eqref{eq:aaa}}{=}&
    \norm{\mR_k - 	 \frac{\mR_k \vs_k \vs_k^\top \mR_k^\top}{\vs_k^\top \mR_k^\top \mR_k \vs_k}\mR_k }_F^2\\
    =& 
    \norm{\mR_k}_F^2 + \frac{\tr\left( \mR_k \vs_k \vs_k^\top   (\mR_k^\top \mR_k)^2 \vs_k\vs_k^\top \mR_k^\top \right)}{(\vs_k^\top \mR_k^\top \mR_k \vs_k)^2} 
    - \frac{2\tr\left( \mR_k^\top \mR_k\vs_k\vs_k^\top \mR_k^\top \mR_k \right)}{\vs_k^\top \mR_k^\top \mR_k \vs_k}\\
    =&
    \norm{\mR_k}_F^2 + \frac{\vs_k^\top   (\mR_k^\top \mR_k)^2 \vs_k}{\vs_k^\top \mR_k^\top \mR_k \vs_k} - 2\frac{\vs_k^\top   (\mR_k^\top \mR_k)^2 \vs_k}{\vs_k^\top \mR_k^\top \mR_k \vs_k}\\
    =&
    \norm{\mR_k}_F^2 - \frac{\vs_k^\top   (\mR_k^\top \mR_k)^2 \vs_k}{\vs_k^\top \mR_k^\top \mR_k \vs_k}\\
    \le&
    \norm{\mR_k}_F^2 - \frac{\vs_k^\top   (\mR_k^\top \mR_k) \vs_k}{\vs_k^\top  \vs_k},
\end{align*}
For the greedy method, it holds that $\frac{\vs_k^\top   (\mR_k^\top \mR_k) \vs_k}{\vs_k^\top  \vs_k} \geq \frac{\norm{\mR_k}_F^2}{n}$.
For the random method, it holds that $\EE_{\vs_k} \left[ \frac{\vs_k^\top   (\mR_k^\top \mR_k) \vs_k}{\vs_k^\top  \vs_k} \right] = \frac{\norm{\mR_k}_F^2}{n}$.
Thus, it holds that
\begin{align*}
    &\EE_{\vs_k}\left[ \norm{\mB_{k+1} - \mJ_k}_F^2  \right] \leq \left(1-\frac{1}{n}\right) \norm{\mR_k}_F^2\\
    \stackrel{(a)}{\Rightarrow}&
    \EE_{\vs_k}\left[\norm{\mB_{k+1} - \mJ_k}_F\right] \leq \left(1 - \frac{1}{2n}\right) \norm{\mR_k}_F\\
    \Rightarrow&
    \EE_{\vs_k}\left[\norm{\mB_{k+1} - \mJ_*}_F\right] \leq \left(1 - \frac{1}{2n}\right) \norm{\mB_k - \mJ_*}_F + \left(1+1-\frac{1}{2n}\right)\norm{\mJ_k - \mJ_*}_F\\
    \stackrel{(b)}{\Rightarrow}&
    \EE_{\vs_k}\left[\norm{\mB_{k+1} - \mJ_*}_F\right] \leq \left(1 - \frac{1}{2n}\right) \norm{\mB_k - \mJ_*}_F + \left(2-\frac{1}{2n}\right)Mr_k,
\end{align*}
where $\stackrel{(a)}{\Rightarrow}$ is because of the Jensen's inequality and $\stackrel{(b)}{\Rightarrow}$ is because of Assumption~\ref{ass:lisp}.
	
For the case $\mR_k \vs_k = \bm{0}$, by the update rule of Algorithm~\ref{algo:aaa-greedy-random}, similar to the proof of Theorem~\ref{thm:main2}, we can obtain that $\mR_k = \bm{0}$. 
Thus, we have $\norm{\mB_{k+1} - \mJ_k}_F = \norm{\mB_k - \mJ_k}_F 
= \bm{0}$, which implies that
\begin{align*}
    \norm{ \mB_{k+1} - \mJ_* }_F \leq \norm{\mB_{k+1} - \mJ_k }_F + \norm{\mJ_k - \mJ_*}_F 
    \leq Mr_k.
\end{align*}
Thus, Eq.~\eqref{eq:ssig} also holds for the case $\mR_k \vs_k = \bm{0}$. 


\end{proof}

\begin{lemma}\label{lem:rr}
Suppose the initialization satisfies $ncMr_0 \leq 1/24$ and $c(\sigma_0 + 4Mr_0) \leq 1/3$. Then we have the following properties.
\begin{enumerate}
    \item For the greedy version, it holds that
    \begin{equation}
        r_{k+1} \leq \frac{3}{2}\left(1 -\frac{1}{4n}\right)^k c\left(\sigma_0 + 3M r_0\right) \cdot r_k. \label{eq:rr_a}
    \end{equation}
    \item For the random version, given $0<\delta<1$, it holds with a probability at least $1 - \delta$ that for all $k \geq 0$, 
    \begin{equation}
        r_{k+1} \leq 24n^2\delta^{-1}\left(1 - \frac{1}{4n+1}\right)^k c\left(\sigma_0 + 3M r_0\right) \cdot r_k. \label{eq:rr_b}
    \end{equation}
\end{enumerate}
\end{lemma}

\begin{proof}
First, by the initialization condition, Eq.~\eqref{eq:rr_dec} always holds. Then it holds that $ncMr_k \leq ncMr_0 \leq 1/24$ for all $k\geq 0$.
Thus, we can obtain that
\begin{align*}
    \EE_{\vs_k}\left[\sigma_{k+1} + 3M r_{k+1}\right] 
    \stackrel{\eqref{eq:ssig}\eqref{eq:rr}}{\leq}&
    \left(1 - \frac{1}{2n}\right) \sigma_k + \left(2-\frac{1}{2n}\right)Mr_k + \frac{3}{2}\left(3cMr_k/2+c\sigma_k \right) 3Mr_k \\
    =&
    \left(1 - \frac{1}{2n} + \frac{9cMr_k}{2}\right)\sigma_k
    + 3Mr_k \left(\frac{2}{3} - \frac{1}{6n}+ \frac{9cMr_k}{4}\right)\\
    \stackrel{(a)}{\leq}&
    \left(1 - \frac{1}{2n} + \frac{9}{48n}\right)\sigma_k + \left(1-\frac{1}{3} - \frac{1}{6n} + \frac{9}{96n}\right)3M r_k \\
    \leq& \left(1 - \frac{1}{4n}\right)\left(\sigma_k +3Mr_k\right),
\end{align*}
where $(a)$ follows from $ncMr_k \leq 1/24$.
Then,
\begin{align}
    \EE\sigma_k \leq \EE\left[\sigma_k+3Mr_k\right] \leq \left(1 -\frac{1}{4n}\right)^k \left(\sigma_0 + 3M r_0\right). \label{eq:sig}
\end{align}

First, for the greedy method, because there is no randomness in the greedy method, Eq.~\eqref{eq:sig} holds without expectation. Combining with Eq.~\eqref{eq:rr_dec} and Eq.~\eqref{eq:rr}, we can obtain that
\begin{align*}
    r_{k+1} 
    \stackrel{\eqref{eq:rr}}{\leq} 
    \frac{3}{2}\left(3cMr_k/2+c\sigma_k \right) r_k
    \leq \frac{3}{2}\cdot c\left(3Mr_k+\sigma_k \right) r_k
    \leq \frac{3}{2}\left(1 -\frac{1}{4n}\right)^k c\left(\sigma_0 + 3M r_0\right) \cdot r_k.
\end{align*}
	
For the random case, using Eq.~\eqref{eq:sig} with \cite[Lemma 26]{lin2021faster}, we get with the probability at least $1 -\delta$,
\begin{equation*}
    \sigma_k + 3M r_k \leq 16n^2\delta^{-1}\left(\sigma_0 + 3M r_0\right)\left(1 - \frac{1}{4n+1}\right)^k, \forall k \geq 0.
\end{equation*}
Thus, we obtain
\begin{align*}
    r_{k+1} 
    \stackrel{\eqref{eq:rr}}{\leq} 
    \frac{3}{2}\left(3cMr_k/2+c\sigma_k \right) r_k
    \leq \frac{3}{2} \cdot c\left(3Mr_k+\sigma_k \right) r_k
    \leq 24n^2\delta^{-1}\left(1 - \frac{1}{4n+1}\right)^k c\left(\sigma_0 + 3M r_0\right)\cdot r_k.
\end{align*}
\end{proof}

\begin{lemma}
The update of Algorithm~\ref{algo:aaa-greedy-random} has the following property:
\begin{align}
    \sigma_{k+1} \leq& \sigma_k + Mr_k. \label{eq:ss_dec}
\end{align}
Furthermore, we have the following properties:
\begin{enumerate}
    \item {\bf For the greedy method}, given $k\geq 0$, if there exist $i, j \in\{0, \dots, n-1\}$ with $i<j$ such that $\vs_{k+i} = \vs_{k+j}$, then
    \begin{equation}
    \sigma_{k+j} \leq \sqrt{n}\sum_{t=k+i}^{k+j} M r_t + 2\sqrt{n}M r_{k+j}, \label{eq:sig_rs}
    \end{equation}
    otherwise, the following inequality holds:
    \begin{equation}
    \sigma_{k+n} \leq \sqrt{n}  \sum_{t=k}^{k+n-1} Mr_t.\label{eq:sig_rs_1}
    \end{equation} 
    \item {\bf For the random method}, given $k\ge 0$, then it holds that with a probability at least $1-\delta$, 
    \begin{equation}
        \sigma_{k+n} \leq 4c_1\delta^{-1}\left(n + 4\log(2n/\delta)\right) \sum_{t=k}^{k+n-1} Mr_t, \text{ if } n \geq n_0 := \log_{c_2}(\delta/4), \label{eq:sig_rs_2}
    \end{equation}
    where $c_1>1, c_2 \in (0, 1)$ are all constants following Lemma \ref{lem:gauss_inv}.
\end{enumerate}
\end{lemma}
\begin{proof}
By the update rule in Eq.~\eqref{eq:aaa}, we can obtain that
\begin{equation}
\mB_{k+1} - \mJ_* = \left(\mI_n - \frac{\mR_k\vs_k\vs_k^\top \mR_k^\top}{\vs_k^\top \mR_k^\top\mR_k\vs_k}\right) 
\left(\mB_k - \mJ_*\right) 
+ \frac{\mR_k\vs_k\vs_k^\top \mR_k^\top}{\vs_k^\top \mR_k^\top\mR_k\vs_k} (\mJ_k - \mJ_*). \label{eq:aaa_up_2}
\end{equation}
Taking the spectral norm on both sides of the above equation, we can obtain that
\begin{align*}
\sigma_{k+1} = \norm{\mB_{k+1} - \mJ_*}_F 
\leq& 
\norm{ \mI_n - \frac{\mR_k\vs_k\vs_k^\top \mR_k^\top}{\vs_k^\top \mR_k^\top\mR_k\vs_k} } \cdot \norm{\mB_k -\mJ_*}_F 
+ \norm{ \frac{\mR_k\vs_k\vs_k^\top \mR_k^\top}{\vs_k^\top \mR_k^\top\mR_k\vs_k} } \cdot \norm{\mJ_k - \mJ_*}_F\\
\leq&
\norm{\mB_k -\mJ_*}_F + \norm{\mJ_k - \mJ_*}_F
\leq
\sigma_k + M r_k.
\end{align*}
Similarly, for any vector $\vs\in\RR^n$, it holds that
\begin{equation*}
\begin{aligned}
    \norm{(\mB_{k+1} - \mJ_*)\vs} 
    \stackrel{\eqref{eq:aaa_up_2}}{=}&
    \norm{\left(\mI_n - \frac{\mR_k\vs_k\vs_k^\top \mR_k^\top}{\vs_k^\top \mR_k^\top\mR_k\vs_k}\right) 
        \left(\mB_k - \mJ_*\right)\vs 
        + \frac{\mR_k\vs_k\vs_k^\top \mR_k^\top}{\vs_k^\top \mR_k^\top\mR_k\vs_k} (\mJ_k - \mJ_*)\vs  }\\
    \leq&
    \norm{ \left(\mB_k - \mJ_*\right)\vs } + \norm{(\mJ_k - \mJ_*)\vs}.
\end{aligned}
\end{equation*}
Repeating the above inequality from the step $i+1$ to $k$ with $i \leq k-1$ and noting that $ (\mB_{i+1} - \mJ_*)\vs_i 
\stackrel{\eqref{eq:null}}{=} (\mJ_i - \mJ_*) \vs_i $ for any $i \geq 0$, we could obtain
\begin{align}
\norm{(\mB_{k+1} - \mJ_*)\vs_i}  
&\leq \norm{(\mB_{i+1} - \mJ_*)\vs_i} + 
\sum_{t=i+1}^{k} \norm{(\mJ_t - \mJ_*)\vs_i} = \norm{(\mJ_{i} - \mJ_*) \vs_i} + \sum_{t=i+1}^{k} \norm{(\mJ_t - \mJ_*)\vs_i} \nonumber \\
&= 
\sum_{t=i}^{k} \norm{(\mJ_t - \mJ_*)\vs_i}
\leq
\sum_{t=i}^{k} Mr_t \norm{\vs_i}. \label{eq:ss}
\end{align}
The above inequality also holds for $i=k$ since $ \norm{(\mB_{k+1} - \mJ_*)\vs_k} 
\stackrel{\eqref{eq:null}}{=} \norm{(\mJ_k - \mJ_*) \vs_k } \leq Mr_k\norm{\vs_k}$.

{\bf For the greedy method}, if there exist $i, j \in\{0, \dots, n-1\}$ with $i<j$ such that $\vs_{k+i} = \vs_{k+j}$, it holds that
\begin{align*}
\sigma_{k+j} 
=& 
\norm{ \mB_{k+j} -\mJ_* }_F
\leq
\sqrt{n}\max_{\ve\in\{\ve_1,\dots,\ve_n\}} \norm{(\mB_{k+j} - \mJ_*)\ve}\\
\le&
\sqrt{n}\max_{\ve\in\{\ve_1,\dots,\ve_n\}} \norm{(\mB_{k+j} - \mJ_{k+j})\ve} + \sqrt{n}\norm{ \mJ_{k+j} - \mJ_* }\\
\stackrel{(a)}{=}&
\sqrt{n} \norm{(\mB_{k+j} - \mJ_{k+j})\vs_{k+j}} + \sqrt{n}\norm{ \mJ_{k+j} - \mJ_*}\\
\stackrel{(b)}{=}&
\sqrt{n} \norm{(\mB_{k+j} - \mJ_{k+j})\vs_{k+i}} + \sqrt{n}\norm{ \mJ_{k+j} - \mJ_* } \\
\leq&\sqrt{n} \norm{(\mB_{k+j} - \mJ_*)\vs_{k+i}} + 2\sqrt{n}\norm{ \mJ_{k+j} - \mJ_* }
\stackrel{\eqref{eq:ss}}{\leq}
\sqrt{n} \sum_{t=k+i}^{k+j-1} Mr_t + 2\sqrt{n} M r_{k+j},
\end{align*}
where $\stackrel{(a)}{=}$ is because of the greedy method and $\vs_{k+i}$ is the coordinate vector which maximizes $\norm{(\mB_{k+i} - \mJ_{k+i})\ve}$ and $\stackrel{(b)}{=}$ is because of $\vs_{k+i} = \vs_{k}$ by assumption.

Otherwise, for any $ 0 \leq i < j \leq n-1, \vs_{k+i} \neq \vs_{k+j}$. Since $\vs_{k},\dots, \vs_{k+n-1} \in \{\ve_1, \dots, \ve_n\}$, we could obtain that $\{\vs_{k}, \dots, \vs_{k+n-1}\} = \{\ve_1, \dots, \ve_n\}$.
Hence, we derive that 
\begin{equation*}
\begin{aligned}
    \sigma_{k+n} &= \norm{(\mB_{k+n} - \mJ_*)}_F = \sqrt{\sum_{i=0}^n \norm{(\mB_{k+n} - \mJ_*) \vs_{k+i}}^2} \leq \sqrt{n} \max_{0\leq i \leq n-1}  \norm{(\mB_{k+n} - \mJ_*) \vs_{k+i}} \\ &\stackrel{\eqref{eq:ss}}{\leq} \sqrt{n} \max_{0 \leq i \leq n-1}\sum_{t=k+i}^{k+n-1} M r_t\norm{\vs_{k+i}} \leq \sqrt{n} \sum_{t=k}^{k+n-1} M r_t.
\end{aligned}
\end{equation*}
Thus, we can obtain Eq.~\eqref{eq:sig_rs_1}.

{\bf For the random method}, denoting $\mS_k = [\vs_k,\dots, \vs_{k+n-1}]$, it holds that 
\begin{equation*}
\begin{aligned}
    \norm{\mB_{k+n} -\mJ_*}_F 
    \leq&
    \norm{\mS_k^{-1}} \cdot  \norm{ \left(\mB_{k+n} -\mJ_*\right) \mS_k }_F 
    =
    \norm{\mS_k^{-1}} \cdot \sqrt{ \sum_{i=k}^{k+n-1} \norm{ \left(\mB_{k+n} -\mJ_*\right) \vs_i }^2 }\\
    \stackrel{\eqref{eq:ss}}{\leq}&
    \norm{\mS_k^{-1}}  \sqrt{ \sum_{i=k}^{k+n-1} \left(\sum_{t=i}^{k+n-1} Mr_i \norm{\vs_i} \right)^2} \leq \norm{\mS_k^{-1}} \sum_{t=k}^{k+n-1} Mr_t \cdot \max_{0 \leq i \leq n-1}{\norm{\vs_{i+k}}}.
\end{aligned}
\end{equation*}
Based on the standard Chi-square concentration inequality \cite[Eq.~(2.18)]{wainwright2019high}, we get
\begin{equation}
    \sP(\norm{\vs_k}^2 \geq (t+1)n) \leq e^{-nt/8}.
\end{equation}
Hence, we get
\begin{equation}
    \sP\left(\exists 0 \leq i \leq n-1, \norm{\vs_i}^2\geq 8\log (2n/\delta) + n\right) \leq n \cdot \frac{\delta}{2n} = \frac{\delta}{2}.
\end{equation}
Moreover, from Lemma \ref{lem:gauss_inv}, we further have that
\begin{equation}
    \sP\left(\frac{1}{\norm{\mS_k^{-1}}} \leq \frac{\delta}{4c_1 \sqrt{n}}\right) \leq \delta/2, \text{if } n \geq n_0 := \log_{c_2}(\delta/4),
\end{equation}
where $c_1>1, c_2 \in (0, 1)$ are all constants in our setting following Lemma \ref{lem:gauss_inv}.

Combining the above inequalities, we could obtain that if $n \geq n_0$, then with a probability at least $1-\delta$, 
\[ \norm{\mB_{k+n} -\mJ_*}_F \leq \frac{4c_1\sqrt{n}}{\delta} \cdot \sqrt{8\log(2n/\delta)+n} \sum_{t=k}^{k+n-1} Mr_t \leq 4c_1\delta^{-1}\left(n + 4\log(2n/\delta)\right) \cdot  \sum_{t=k}^{k+n-1} Mr_t. \]
\end{proof}

Combining the above lemmas, we can obtain the final convergence rate of our algorithm.

\begin{thm}\label{thm:main}
Supposing the initialization of Algorithm~\ref{algo:aaa-greedy-random} satisfies the one in Lemma~\ref{lem:rr}, then the algorithm has the following convergence properties:
\begin{enumerate}
\item {\bf For the greedy method}, given $k\geq 0$, if there exist $i, j \in\{0, \dots, n-1\}$ with $i<j$ such that $\vs_{k+i} = \vs_{k+j}$, then
\begin{equation}
    r_{k+n+1} \leq r_{k+j+1} \leq 6c\sqrt{n} M \cdot \left(\prod_{t=k+i}^{k+j-1} \cdot \rho_t\right) r^2_{k+i} \leq 6c\sqrt{n} M \cdot \left(\prod_{t=k+i}^{k+j-1} \cdot \rho_t\right) r^2_{k}, \label{eq:main1}
\end{equation}
with $\rho_t$ defined as
\begin{equation} \label{eq:rho}
\rho_t = \frac{3}{2}\left(1 -\frac{1}{4n}\right)^t c\left(\sigma_0 + 3M r_0\right).
\end{equation}
Otherwise, it holds that
\begin{equation}\label{eq:main2}
    r_{k+n+1} \leq 3c\sqrt{n}M \left(\prod_{t=k}^{k+n-1} \rho_t\right) r_k^2.
\end{equation}
\item {\bf For the random method}, given $k\geq 0$, and $0<\delta<1$, then it holds that with a probability at least $1-2\delta$
\begin{equation}\label{eq:rand}
r_{k+n+1} \leq 12cc_1\delta^{-1}\left(n + 4\log(2n/\delta)\right)M \left(\prod_{t=k}^{k+n-1} \rho'_t\right) r_k^2, \text{ if } n \geq n_0 := \log_{c_2}(\delta/4),
\end{equation}
with $\rho_t'$ defined as
\begin{equation}
    \rho_t' = \min\left\{24n^2\delta^{-1}\left(1 - \frac{1}{4n+1}\right)^k c\left(\sigma_0 + 3M r_0\right), \frac{1}{2}\right\}.
\end{equation}
\end{enumerate}
\end{thm}
\begin{proof}
{\bf For the greedy method},
by Eq.~\eqref{eq:rr_dec} and Eq.~\eqref{eq:rr_a}, it holds that for any $i > 0$, 
\begin{equation}
r_{k+i} \leq \prod_{t=k}^{k+i-1}\rho_t \cdot r_{k}. \label{eq:rr_f}
\end{equation}
if there exists $i, j \in\{0, \dots, n-1\}$ with $i<j$ such that $\vs_{k+i} = \vs_{k+j}$, then
by Eq.~\eqref{eq:rr_dec}, we can get $r_{k+j} \leq r_{k+i}\cdot 2^{i-j}$, leading to 
\begin{align*}
r_{k+j+1} 
\stackrel{\eqref{eq:rr}}{\leq}&
1.5c(\sigma_{k+j} + 1.5 Mr_{k+j})r_{k+j}
\stackrel{\eqref{eq:sig_rs}}{\leq}
1.5 c\left(\sqrt{n}\sum_{t=k+i}^{k+j-1} Mr_t + (2\sqrt{n}+1.5) Mr_{k+j} \right) r_{k+j}\\
\leq&
1.5c\left(\sqrt{n}\sum_{t=0}^{j-i-1} 2^{-t} + (2\sqrt{n}+1.5) 2^{i-j} \right) Mr_{k+i}r_{k+j} \leq 1.5c\left(2\sqrt{n}+1.5\right) Mr_{k+i}r_{k+j} 
\\
\leq& 6c\sqrt{n} M r_{k+i}r_{k+j}
\stackrel{\eqref{eq:rr_f}}{\leq}
6c\sqrt{n} M \left(\prod_{t=k+i}^{k+j-1}\rho_t\right) \cdot r_{k+i}^2.
\end{align*}
Thus we obtain the result in Eq.~\eqref{eq:main1} since $r_{k+n+1} \leq r_{k+j+1}$ and $r_{k+i} \leq r_k$ by Eq.~\eqref{eq:rr_dec}.

Otherwise, we can prove Eq.~\eqref{eq:main2} through
\begin{align*}
r_{k+n+1} 
\stackrel{\eqref{eq:rr}}{\leq}&
1.5c(\sigma_{k+n} + 1.5Mr_{k+n})r_{k+n}
\stackrel{\eqref{eq:sig_rs_1}}{\leq}
1.5c\left(\sqrt{n}\sum_{t=k}^{k+n-1} Mr_t + 1.5Mr_{k+n}\right) r_{k+n}\\
\leq&
1.5c\left(\sqrt{n}\sum_{t=0}^{n-1} 2^{-t} + 1.5\cdot 2^{-n}\right) Mr_{k} r_{k+n} \leq 3c\sqrt{n} M r_{k}r_{k+n}
\stackrel{\eqref{eq:rr_f}}{\leq}
3c\sqrt{n} M \left(\prod_{t=k}^{k+n-1}\rho_t\right) \cdot r_{k}^2.
\end{align*}

{\bf For the  random method}, 
by Eq.~\eqref{eq:rr_dec}, Eq.~\eqref{eq:rr_b} and definition of $\rho_t'$, it holds that with a probability at least $1-\delta$ 
\begin{equation}
r_{k+n} \leq \prod_{t=k}^{k+n-1}\rho'_t \cdot r_{k}. \label{eq:rr_f1}
\end{equation}
Denoting $\theta = 4c_1\delta^{-1}\left(n + 4\log(2n/\delta)\right) > 1$, we can obtain that with a probability at least $1-2\delta$,
\begin{align*}
r_{k+n+1} 	
\stackrel{\eqref{eq:rr}}{\leq}&
1.5c(\sigma_{k+n} + 1.5Mr_{k+n})r_{k+n}
\stackrel{\eqref{eq:sig_rs_2}}{\leq}
1.5c\left( \theta \sum_{t=k}^{k+n-1} Mr_t + 1.5Mr_{k+n} \right) r_{k+n}\\
\leq& 1.5c\left(\theta \sum_{t=0}^{n-1} 2^{-t} + 1.5\cdot 2^{-n}\right) Mr_{k} r_{k+n}
\leq 
3c\theta M r_k r_{k+n}
\stackrel{\eqref{eq:rr_f1}}{\leq}
3c\theta M\left(\prod_{t=k}^{k+n-1} \rho'_t\right) r_k^2.
\end{align*}
Therefore, Eq.~\eqref{eq:rand} holds.
\end{proof}

\begin{remark}
Theorem~\ref{thm:main} shows that our algorithm with the greedy or random method can achieve a $n$-step super quadratic convergence for every $n$ steps.
Using Eq.~\eqref{eq:main2} as an example, it holds that 
\begin{align*}
\lim_{k\to\infty} \frac{r_{k+n+1}}{r_k^2} 
= 
\lim_{k\to\infty} 3c\sqrt{n}M \left(\prod_{t=k}^{k+n-1}\rho_t\right)
\stackrel{\eqref{eq:rho}}{\leq}
\lim_{k\to\infty} 3c\sqrt{n}M \left(\frac{3}{2}\left(1 -\frac{1}{4n}\right)^k c\left(\sigma_0 + 3M r_0\right)\right)^{n}
=0.
\end{align*}
Comparing the greedy method with the random method, we can observe that the greedy method may outperform the random one since there may be $i, j \in\{0, \dots, n-1\}$ with $i<j$ such that $\vs_{k+i} = \vs_{k+j}$.
In this case, the super quadratic convergence happens in $i$ steps which are less than $n$ steps.
\end{remark}

\begin{remark}
	Remark~\ref{rmk:quad} discusses that with the restart strategy, the conjugate gradient method and AA-I can achieve $n$-step quadratic convergence rates.
	In contrast, our algorithm does \emph{not} require a restart.
	Furthermore, Theorem~\ref{thm:main} shows that given any $k$, in the following $n+1$ steps, there must happen a \emph{super} quadratic convergence in the update of our algorithm. 
	This property does \emph{not} hold for the  the conjugate gradient method  with a restart strategy.
\end{remark}

\begin{remark}
Although we only prove our methods in well-defined Jacobian cases, that is, Assumption~\ref{ass:lisp} holds, we believe our method could also apply to generalized Jacobian cases as well, such as the famous semi-smooth settings with the generalized Hessian matrix \cite{xiao2018regularized,hu2022local}.
Noting that the Jacobian is well-defined in most problems in the neighborhood of the solution, we could apply our method with the generalized Jacobian to solve some ill-defined problems starting from proper initialization.
\end{remark}

\section{Numerical Experiments}

In Section~\ref{sec:aaa}, we propose our algorithm and provide the explicit convergence rate of our algorithm. 
In this section, we will validate the super quadratic convergence rate of our algorithm by experiments.
We will conduct the experiment on the widely used logistic regression and elastic net regression.
Furthermore, we will compare our algorithm with the \texttt{AA-I} \citep{fang2009two}, \texttt{AA-II} \citep{anderson1965iterative} and \texttt{AA-Safe} \citep{zhang2020globally}. 
For these variants of Anderson's acceleration, we choose memory $m = 10$ and initialize $\mB_0 = \mJ(\vx_0)$ with $\vx_0$ being a Gaussian random vector normalized to $\norm{\vx_0} = 1$.
The hyper-parameters of \texttt{AA-safe} follow from the setting of \citet{zhang2020globally}.

\paragraph{Regularized Logistic Regression.}
We consider the following regularized logistic regression (Reg-Log-Exp) problem:
\begin{equation*}
    \min_{\vx\in\sR^n} f(\vx) := \frac{1}{m}\sum_{i=1}^m \log\left(1+e^{b_i \vx^\top \va_i}\right)+\frac{\mu}{2}\norm{\vx}_2^2.
\end{equation*}
We use UCI Mushrooms dataset where $m=8124, n=112$, and scale data point $\mA = [\va_1\dots, \va_m]^\top$ with $\norm{\va_i}=1, b_i=\pm{1}$. We adopt
\[ \mu = 0.01, \norm{\vx_0}=1, \eta = \frac{2}{L+\mu}, \mbox{ with } L:=\norm{\mA}_2^2/(4m), \]
and solve the non-linear equation problem:
\[  \mF(\vx) = \vx - G(\vx), \mbox{ and } G(\vx) := \vx - \eta \nabla f(\vx). \]
In this problem, it holds that $\mJ_k = \eta \nabla^2 f(\vx_k)$ which is a positive definite matrix.
However, in our algorithm, $\mB_k$ is commonly asymmetric.

\begin{figure}[t]
	\centering
	\begin{subfigure}[b]{0.49\textwidth}
		\includegraphics[width=\linewidth]{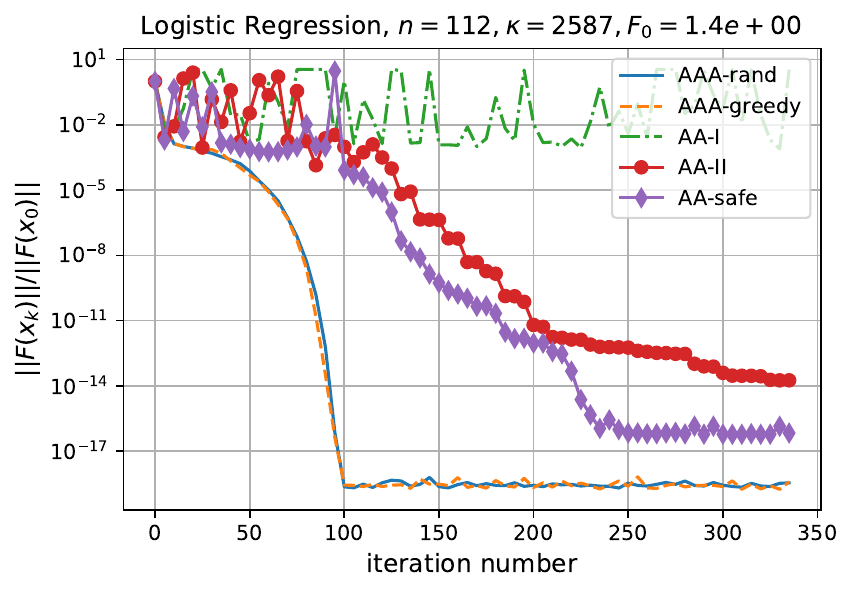}
		\caption{Regularized Logistic Regression.}\label{subfig:a}
	\end{subfigure}
	\begin{subfigure}[b]{0.48\textwidth}
		\includegraphics[width=\linewidth]{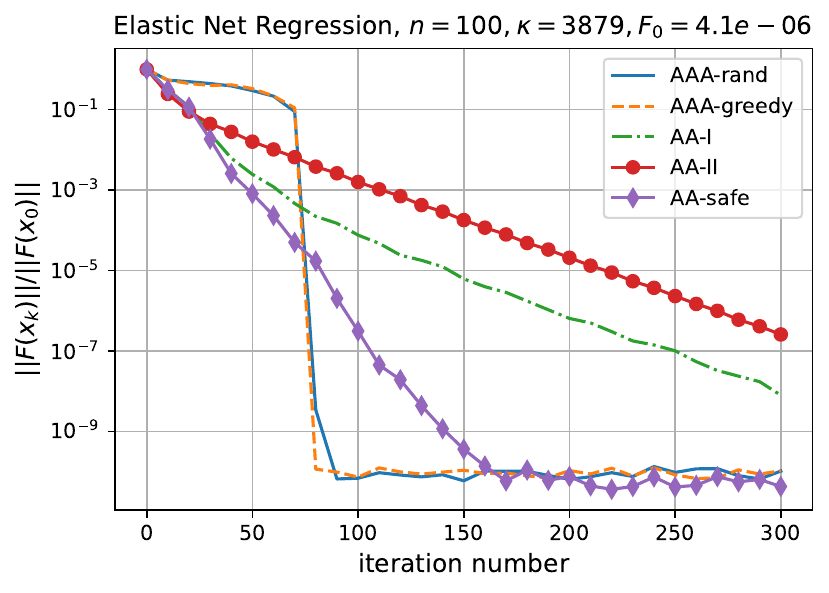}
		\caption{Elastic net regression.} \label{subfig:b}
	\end{subfigure}
	\caption{Numerical experiments.}
	\label{fig:aa-exp}
\end{figure}

\paragraph{Elastic net regression.}
We consider the following elastic net regression (ENR) problem:
\begin{equation*}
    \min_{\vx\in\sR^n} f(\vx) := \frac{1}{2}\norm{\mA\vx-\vb}_2^2 + \mu\left(\frac{1}{4}\norm{\vx}_2^2+\frac{1}{2}\norm{\vx}_1\right),
\end{equation*}
where $\mA\in\RR^{m\times n}$ and $\vb\in\RR^n$ with $m=100, n=100$, and $\mu = 0.001 \mu_{\max}$ with $\mu_{\max} = \norm{\mA^\top \vb}_\infty$ being the smallest value under which the ENR problem admits only the zero solution \citep{o2016conic}.
In our experiment, we use a Gaussian random matrix $\mA$ with each entry $\mA_{i,j} \sim \cN(0, 1)$, and  $\hat{\vx} \in\RR^n$ is a sparse vector with sparsity $0.1$ whose non-zero entries are chosen from the Gaussian distribution. 
We then generate $\vb$ as $\vb = \mA\hat{\vx} + 0.1\vw$, where $\vw$ is also a random Gaussian vector.

Applying ISTA method \citep{daubechies2004iterative} to the ENR problem, we obtain the following iteration scheme:
\begin{equation*}
	\vx_{k+1} = \gS_{\alpha\mu/2}\left(\vx_k - \eta \left(\mA^\top (\mA\vx_k -\vb) + \frac{\mu}{2} \vx_k\right) \right),
\end{equation*}
in which we choose $\eta = 1.8/L$ with $L = \norm{\mA}^2 + \mu/2$ and $\gS_{a}(x)$ is the shrinkage operator, i.e., for $x\in\RR$, it holds that, 
\[ \gS_{a}(x) = \mathrm{sign}(x) \left(|x|-a\right)_+. \]
Accordingly, we try to solve the following non-linear equation problem:
\[ \mF(\vx) = \vx - G(\vx), \mbox{ with } G(\vx) := \gS_{\alpha\mu/2}\left(\vx - \alpha \left(\mA^\top(\mA \vx - \vb) + \frac{\mu}{2} \vx\right)\right).  \]

The experimental results for the previously mentioned problems are depicted in Figure~\ref{fig:aa-exp}. 
The results demonstrate that both the greedy and random versions of our algorithm successfully identify the optimum within $ n$ steps for logistic and elastic net regression.
Figure~\ref{fig:aa-exp} distinctly showcases the super quadratic convergence rate achieved by our \algname. 
This observation corroborates the convergence analysis detailed in Section~\ref{sec:aaa}. 
Additionally, although the super quadratic convergence rate of \algname~is guaranteed under the condition that the problem's Jacobian is Lipschitz continuous, Figure\ref{subfig:b} illustrates that \algname~still attains a rapid convergence rate. 
This could be attributed to the Jacobian being well-defined near the solution for elastic net regression. 
Employing a generalized Jacobian, our method overcomes the challenge of ill-defined problems (lacking Lipschitz continuity) through appropriate initialization, highlighting the robustness and effectiveness of \algname~in achieving fast convergence rates under complex conditions.

\section{Conclusion}

This study introduces a new approach to Anderson's acceleration, designated as \algname, marking a novel iteration of type-I Anderson's acceleration. 
It can also be considered a variation of the Broyden's good method. 
Our algorithm achieves numerical stability without the need for a restart strategy. 
Unlike the existing \texttt{AA-I} variants and quasi-Newton methods that attain a superlinear convergence rate, our \algname~exhibits an $n$-step super quadratic convergence rate, for which we detail the explicit convergence rates for both greedy and random versions. 
These convergence rates introduce new insights into the theories underlying \texttt{AA-I} and quasi-Newton methods, aiding in a deeper understanding of the properties of Anderson's acceleration and quasi-Newton methods. Additionally, our experimental results confirm the accelerated convergence rate of our \algname~method, showcasing its effectiveness and potential to advance the field of numerical optimization.

\appendix

\section{Useful Lemmas}


\begin{lemma}[Theorem 1.2 of \citet{rudelson2008littlewood}]\label{lem:gauss_inv}
Let $\xi_1, \dots, \xi_n$ be independent centered real random variables with variances at least $1$ and subgaussian moments bounded by $B$. Let $\mA$ be an $n\times n$ matrix whose rows are independent copies of
the random vector $(\xi_1, \dots, \xi_n)$. Then for every $\epsilon>0$ one has
\begin{equation}
    \sP\left(s_n(\mA) \leq \epsilon \sqrt{n}\right) \leq c_1\epsilon + c_2^n.
\end{equation}
where $s_n(\mA)=\inf_{\vx: \norm{\vx}=1}\norm{\mA\vx}$, $c_1 > 1$ and $c_2 \in (0, 1)$ depend (polynomially) only on $B$.
\end{lemma}

\bibliographystyle{plainnat}
\bibliography{reference}
\end{document}